\documentclass[a4paper,10pt,oneside]{amsart}

\usepackage[utf8]{inputenc}
\usepackage{url}
\usepackage[hidelinks]{hyperref}

\usepackage[DIV=10]{typearea}

\usepackage{microtype}
\usepackage{amssymb,amsmath,amsopn,amsxtra,amsthm,amsfonts}
\usepackage{MnSymbol}
\usepackage[mathcal]{euscript}
\let\mathscr\mathcal
\usepackage[bb=px]{mathalfa}

\usepackage{enumitem}
\setlist[enumerate,1]{label={(\arabic*)},itemsep=\parskip} 
\setlist[itemize,1]{itemsep=\parskip} 
\newlist{thmlist}{enumerate}{2}
\setlist[thmlist,1]{label={\em(\roman*)},ref={(\roman*)},%
  itemsep=\parskip,leftmargin=*,align=left}
\setlist[thmlist,2]{label={\em(\alph*)},ref={(\alph*)},%
  itemsep=\parskip,leftmargin=*,align=left,topsep=0.1cm}
\newlist{defnlist}{enumerate}{2}
\setlist[defnlist,1]{label={(\roman*)},ref={(\roman*)},itemsep=\parskip,%
  leftmargin=*,align=left}
\setlist[defnlist,2]{label={(\alph*)},ref={(\alph*)},itemsep=\parskip,%
  leftmargin=*,align=left,topsep=0.1cm}

\newtheorem{thm}[subsubsection]{Theorem}
\newtheorem{cor}[subsubsection]{Corollary}
\newtheorem{lem}[subsubsection]{Lemma}
\newtheorem{prop}[subsubsection]{Proposition}

\newtheorem{conj}[subsubsection]{Conjecture}

\theoremstyle{definition}
\newtheorem{defn}[subsubsection]{Definition}
\newtheorem{rem}[subsubsection]{Remark}

\newtheorem{exam}[subsubsection]{Example}
\newtheorem{exams}[subsubsection]{Examples}
\newtheorem{constr}[subsubsection]{Construction}

\newcommand{\constrref}[1]{Construction~\ref{#1}}
\newcommand{\thmref}[1]{Theorem~\ref{#1}}

\newcommand{\secref}[1]{Sect.~\ref{#1}}
\newcommand{\ssecref}[1]{Subsect. ~\ref{#1}}

\newcommand{\lemref}[1]{Lemma~\ref{#1}}
\newcommand{\propref}[1]{Proposition~\ref{#1}}
\newcommand{\corref}[1]{Corollary~\ref{#1}}
\newcommand{\conjref}[1]{Conjecture~\ref{#1}}
\newcommand{\remref}[1]{Remark~\ref{#1}}
\newcommand{\defref}[1]{Definition~\ref{#1}}

\renewcommand{\eqref}[1]{(\ref{#1})}

\newcommand{\examref}[1]{Example~\ref{#1}}

\numberwithin{equation}{subsection}

\usepackage{tikz}
\usetikzlibrary{matrix}
\usepackage{tikz-cd}

\newcommand{\changelocaltocdepth}[1]{%
  \addtocontents{toc}{\protect\setcounter{tocdepth}{#1}}%
  \setcounter{tocdepth}{#1}}

\usepackage{xspace}

\usepackage{xifthen}

\usepackage{xparse}

\usepackage{mathtools}

\newcommand{\nc}{\newcommand}
\nc{\renc}{\renewcommand}
\nc{\ssec}{\subsection}
\nc{\sssec}{\subsubsection}
\nc{\on}{\operatorname}
\nc{\term}[1]{#1\xspace}

\nc{\au}{\ensuremath{\mathcal{A}}\xspace}
\nc{\cu}{\ensuremath{\mathcal{C}}\xspace}
\nc{\du}{\ensuremath{\mathcal{D}}\xspace}
\nc{\hu}{\ensuremath{\mathcal{H}}\xspace}
\nc{\eu}{\ensuremath{\mathcal{E}}\xspace}

\nc{\sA}{\ensuremath{\mathcal{A}}\xspace}
\nc{\sB}{\ensuremath{\mathcal{B}}\xspace}
\nc{\sC}{\ensuremath{\mathcal{C}}\xspace}
\nc{\sD}{\ensuremath{\mathcal{D}}\xspace}
\nc{\sE}{\ensuremath{\mathcal{E}}\xspace}
\nc{\sF}{\ensuremath{\mathcal{F}}\xspace}
\nc{\sG}{\ensuremath{\mathcal{G}}\xspace}
\nc{\sH}{\ensuremath{\mathcal{H}}\xspace}
\nc{\sI}{\ensuremath{\mathcal{I}}\xspace}
\nc{\sJ}{\ensuremath{\mathcal{J}}\xspace}
\nc{\sK}{\ensuremath{\mathcal{K}}\xspace}
\nc{\sL}{\ensuremath{\mathcal{L}}\xspace}
\nc{\sM}{\ensuremath{\mathcal{M}}\xspace}
\nc{\sN}{\ensuremath{\mathcal{N}}\xspace}
\nc{\sO}{\ensuremath{\mathcal{O}}\xspace}
\nc{\sP}{\ensuremath{\mathcal{P}}\xspace}
\nc{\sQ}{\ensuremath{\mathcal{Q}}\xspace}
\nc{\sR}{\ensuremath{\mathcal{R}}\xspace}
\nc{\sS}{\ensuremath{\mathcal{S}}\xspace}
\nc{\sT}{\ensuremath{\mathcal{T}}\xspace}
\nc{\sU}{\ensuremath{\mathcal{U}}\xspace}
\nc{\sV}{\ensuremath{\mathcal{V}}\xspace}
\nc{\sW}{\ensuremath{\mathcal{W}}\xspace}
\nc{\sX}{\ensuremath{\mathcal{X}}\xspace}
\nc{\sY}{\ensuremath{\mathcal{Y}}\xspace}
\nc{\sZ}{\ensuremath{\mathcal{Z}}\xspace}

\nc{\bA}{\ensuremath{\mathbf{A}}\xspace}
\nc{\bB}{\ensuremath{\mathbf{B}}\xspace}
\nc{\bC}{\ensuremath{\mathbf{C}}\xspace}
\nc{\bD}{\ensuremath{\mathbf{D}}\xspace}
\nc{\bE}{\ensuremath{\mathbf{E}}\xspace}
\nc{\bF}{\ensuremath{\mathbf{F}}\xspace}
\nc{\bG}{\ensuremath{\mathbf{G}}\xspace}
\nc{\bH}{\ensuremath{\mathbf{H}}\xspace}
\nc{\bI}{\ensuremath{\mathbf{I}}\xspace}
\nc{\bJ}{\ensuremath{\mathbf{J}}\xspace}
\nc{\bK}{\ensuremath{\mathbf{K}}\xspace}
\nc{\bL}{\ensuremath{\mathbf{L}}\xspace}
\nc{\bM}{\ensuremath{\mathbf{M}}\xspace}
\nc{\bN}{\ensuremath{\mathbf{N}}\xspace}
\nc{\bO}{\ensuremath{\mathbf{O}}\xspace}
\nc{\bP}{\ensuremath{\mathbf{P}}\xspace}
\nc{\bQ}{\ensuremath{\mathbf{Q}}\xspace}
\nc{\bR}{\ensuremath{\mathbf{R}}\xspace}
\nc{\bS}{\ensuremath{\mathbf{S}}\xspace}
\nc{\bT}{\ensuremath{\mathbf{T}}\xspace}
\nc{\bU}{\ensuremath{\mathbf{U}}\xspace}
\nc{\bV}{\ensuremath{\mathbf{V}}\xspace}
\nc{\bW}{\ensuremath{\mathbf{W}}\xspace}
\nc{\bX}{\ensuremath{\mathbf{X}}\xspace}
\nc{\bY}{\ensuremath{\mathbf{Y}}\xspace}
\nc{\bZ}{\ensuremath{\mathbf{Z}}\xspace}

\nc{\dA}{\ensuremath{\mathds{A}}\xspace}
\nc{\dB}{\ensuremath{\mathds{B}}\xspace}
\nc{\dC}{\ensuremath{\mathds{C}}\xspace}
\nc{\dD}{\ensuremath{\mathds{D}}\xspace}
\nc{\dE}{\ensuremath{\mathds{E}}\xspace}
\nc{\dF}{\ensuremath{\mathds{F}}\xspace}
\nc{\dG}{\ensuremath{\mathds{G}}\xspace}
\nc{\dH}{\ensuremath{\mathds{H}}\xspace}
\nc{\dI}{\ensuremath{\mathds{I}}\xspace}
\nc{\dJ}{\ensuremath{\mathds{J}}\xspace}
\nc{\dK}{\ensuremath{\mathds{K}}\xspace}
\nc{\dL}{\ensuremath{\mathds{L}}\xspace}
\nc{\dM}{\ensuremath{\mathds{M}}\xspace}
\nc{\dN}{\ensuremath{\mathds{N}}\xspace}
\nc{\dO}{\ensuremath{\mathds{O}}\xspace}
\nc{\dP}{\ensuremath{\mathds{P}}\xspace}
\nc{\dQ}{\ensuremath{\mathds{Q}}\xspace}
\nc{\dR}{\ensuremath{\mathds{R}}\xspace}
\nc{\dS}{\ensuremath{\mathds{S}}\xspace}
\nc{\dT}{\ensuremath{\mathds{T}}\xspace}
\nc{\dU}{\ensuremath{\mathds{U}}\xspace}
\nc{\dV}{\ensuremath{\mathds{V}}\xspace}
\nc{\dW}{\ensuremath{\mathds{W}}\xspace}
\nc{\dX}{\ensuremath{\mathds{X}}\xspace}
\nc{\dY}{\ensuremath{\mathds{Y}}\xspace}
\nc{\dZ}{\ensuremath{\mathds{Z}}\xspace}

\nc{\bbA}{\ensuremath{\mathbb{A}}\xspace}
\nc{\bbB}{\ensuremath{\mathbb{B}}\xspace}
\nc{\bbC}{\ensuremath{\mathbb{C}}\xspace}
\nc{\bbD}{\ensuremath{\mathbb{D}}\xspace}
\nc{\bbE}{\ensuremath{\mathbb{E}}\xspace}
\nc{\bbF}{\ensuremath{\mathbb{F}}\xspace}
\nc{\bbG}{\ensuremath{\mathbb{G}}\xspace}
\nc{\bbH}{\ensuremath{\mathbb{H}}\xspace}
\nc{\bbI}{\ensuremath{\mathbb{I}}\xspace}
\nc{\bbJ}{\ensuremath{\mathbb{J}}\xspace}
\nc{\bbK}{\ensuremath{\mathbb{K}}\xspace}
\nc{\bbL}{\ensuremath{\mathbb{L}}\xspace}
\nc{\bbM}{\ensuremath{\mathbb{M}}\xspace}
\nc{\bbN}{\ensuremath{\mathbb{N}}\xspace}
\nc{\bbO}{\ensuremath{\mathbb{O}}\xspace}
\nc{\bbP}{\ensuremath{\mathbb{P}}\xspace}
\nc{\bbQ}{\ensuremath{\mathbb{Q}}\xspace}
\nc{\bbR}{\ensuremath{\mathbb{R}}\xspace}
\nc{\bbS}{\ensuremath{\mathbb{S}}\xspace}
\nc{\bbT}{\ensuremath{\mathbb{T}}\xspace}
\nc{\bbU}{\ensuremath{\mathbb{U}}\xspace}
\nc{\bbV}{\ensuremath{\mathbb{V}}\xspace}
\nc{\bbW}{\ensuremath{\mathbb{W}}\xspace}
\nc{\bbX}{\ensuremath{\mathbb{X}}\xspace}
\nc{\bbY}{\ensuremath{\mathbb{Y}}\xspace}
\nc{\bbZ}{\ensuremath{\mathbb{Z}}\xspace}

\nc{\mrm}[1]{\ensuremath{\mathrm{#1}}\xspace}
\nc{\mit}[1]{\ensuremath{\mathit{#1}}\xspace}
\nc{\mbf}[1]{\ensuremath{\mathbf{#1}}\xspace}
\nc{\mcal}[1]{\ensuremath{\mathcal{#1}}\xspace}
\nc{\msc}[1]{\ensuremath{\mathscr{#1}}\xspace}
\nc{\mfr}[1]{\ensuremath{\mathfrak{#1}}\xspace}

\renc{\bar}[1]{\overline{#1}}

\let\sectsign\S
\let\S\relax

\nc{\sub}{\subset}
\nc{\too}{\longrightarrow}
\nc{\hook}{\hookrightarrow}
\nc*{\hooklongrightarrow}{\ensuremath{\lhook\joinrel\relbar\joinrel\rightarrow}}
\nc{\hooklong}{\hooklongrightarrow}
\nc{\twoheadlongrightarrow}{\relbar\joinrel\twoheadrightarrow}
\nc{\shiso}{\approx}
\nc{\isoto}{\xrightarrow{\sim}}
\nc{\isofrom}{\xleftarrow{\sim}}
\renc{\ge}{\geqslant}
\renc{\le}{\leqslant}

\nc{\id}{\mathrm{id}}

\DeclareMathOperator{\Hom}{\on{Hom}}
\nc{\uHom}{\underline{\smash{\Hom}}}
\DeclareMathOperator{\Maps}{\on{Maps}}

\DeclareMathOperator{\End}{\on{End}}

\nc{\uEnd}{\underline{\smash{\End}}}

\nc{\colim}{\varinjlim}
\renc{\lim}{\varprojlim}
\nc{\Cofib}{\on{Cofib}}
\nc{\Fib}{\on{Fib}}
\nc{\initial}{\varnothing}
\nc{\op}{\mathrm{op}}

\nc{\Spc}{\mrm{Spc}}
\nc{\Spt}{\mrm{Spt}}
\nc{\Spec}{\on{Spec}}
\nc{\Stk}{\mrm{Stk}}
\nc{\Sch}{\mrm{Sch}}
\nc{\aff}{\mrm{aff}}
\nc{\A}{\mbf{A}}
\renc{\P}{\mbf{P}}
\nc{\cl}{{\mrm{cl}}}
\nc{\bDelta}{\mathbf{\Delta}}
\nc{\un}{\mathbf{1}}
\nc{\Tot}{\on{Tot}}
\nc{\Cech}{\textnormal{\v{C}}}
\nc{\Mod}{\mrm{Mod}}
\nc{\Qcoh}{\on{Qcoh}}
\nc{\free}{\mrm{free}}
\nc{\perf}{\mrm{perf}}
\nc{\aperf}{\mrm{aperf}}
\nc{\coh}{\mrm{coh}}
\nc{\Einfty}{{\sE_\infty}}
\nc{\modmod}{/\!\!/}
\nc{\heart}{\heartsuit}
\nc{\proj}{\mrm{proj}}
\nc{\K}{\on{K}}
\nc{\G}{\on{G}}
\nc{\GL}{\on{GL}}
\nc{\BGL}{\on{BGL}}
\nc{\M}{\on{M}}
\nc{\KH}{\on{KH}}
\nc{\CAlg}{\on{CAlg}}
\nc{\cn}{\mrm{cn}}
\nc{\hw}{\mrm{Hw}}
\nc{\htt}{\mrm{Ht}}
\nc{\Fun}{\on{Fun}}
\nc{\Funadd}{\on{Fun}_{\mrm{add}}}
\nc{\Funex}{\on{Fun}_{\mrm{ex}}}
\nc{\Ind}{\on{Ind}}

\nc{\scr}{\term{simplicial commutative ring}}
\nc{\scrs}{\term{simplicial commutative rings}}

\nc{\Einfring}{\term{$\Einfty$-ring}}
\nc{\Einfrings}{\term{$\Einfty$-rings}}

\nc{\Ering}{\term{$\sE_1$-ring}}
\nc{\Erings}{\term{$\sE_1$-rings}}

\nc{\inftyCat}{\term{$\infty$-category}}
\nc{\inftyCats}{\term{$\infty$-categories}}

\nc{\inftyTop}{\term{$\infty$-topos}}
\nc{\inftyTops}{\term{$\infty$-toposes}}

\nc{\inftyGrpd}{\term{$\infty$-groupoid}}
\nc{\inftyGrpds}{\term{$\infty$-groupoids}}

\title{Regularity of spectral~stacks and discreteness~of~weight-hearts}

\author[V. Sosnilo]{Vladimir Sosnilo}
\address{Laboratory of Modern Algebra and Applications, St. Petersburg State University, 14th line, 29B, 199178 Saint Petersburg, Russia}
\address{St. Petersburg Department of Steklov Mathematical Institute of Russian Academy of Sciences, Fontanka, 27, 191023 Saint Petersburg, Russia}
\email{\href{mailto:vsosnilo@gmail.com}{vsosnilo@gmail.com}}

\begin{document}

\begin{abstract}
We study regularity in the context of connective ring spectra and spectral stacks.
Parallel to that, we construct a weight structure on the category of compact quasi-coherent sheaves on spectral quotient stacks of the
form $X=[\Spec R/G]$ defined over a field, where $R$ is a connective
$\Einfty$-$k$-algebra and $G$ is a linearly reductive group acting on $R$.
Under reasonable assumptions we show that regularity of $X$ is equivalent to regularity of $R$. We also show that if $R$ is bounded, such a stack is discrete.
This result can be interpreted in terms of weight structures and suggests a general phenomenon: for a symmetric monoidal stable \inftyCat
with a compatible bounded weight structure, the existence of an adjacent t-structure satisfying a strong boundedness condition should imply
discreteness of the weight-heart.

We also prove a gluing result for weight structures and adjacent t-structures, in the setting of a semi-orthogonal decomposition of stable \inftyCats.
\end{abstract}

\maketitle

\setcounter{tocdepth}{1}
\parskip 0pt
\tableofcontents
\parskip 0.2cm


\section{Introduction}
\label{sec:intro}


Let $\sC$ be a stable \inftyCat.
The question of existence of a bounded t-structure on $\sC$ has been studied recently by several authors.
In \cite{AGH} it was proved that if such a t-structure exists, the negative K-group $\K_{-1}(\sC)$ must vanish.
In this paper, we are interested in this question under the additional assumption that $\sC$ is equipped with a bounded weight structure.
In this case the vanishing of $\K_{-1}(\sC)$ is equivalent to the vanishing of $\K_{-1}$ of the heart of the weight structure, by \cite{Vovastheoremoftheheart}.
Moreover, any bounded weight structure on $\sC$ gives rise to a natural t-structure on its ind-completion $\Ind(\sC)$.
We may thus ask when this t-structure restricts to $\sC$.
When this is the case, we say the restricted t-structure is \emph{adjacent} to the weight structure.
The simplest example of this phenomenon is for the \inftyCat of perfect complexes over a regular ring, which admits a natural weight structure whose non-negative part is given by the full subcategory of complexes of finitely generated projective modules concentrated in non-negative degrees.

If the heart of the bounded weight structure on $\sC$ is generated by a single object whose endomorphism ring spectrum $R$ is noetherian, then $\sC$ admits an adjacent t-structure if and only if the ring spectrum $R$ is \emph{regular}.
The latter property was studied in \cite{BarwickHeart} and \cite{BarwickLawson} (under the name ``almost regular''), 
and many examples such as $\operatorname{ko}$, $\operatorname{ku}$, $\operatorname{BP}\langle n \rangle$, and $\operatorname{tmf}$ were introduced.
Returning to the original question, one might also wonder whether the adjacent t-structure ever happens to be bounded.
It turns out this can only happen either if $R$ is not quasi-commutative or if $R \simeq \pi_0(R)$ is a discrete regular ring (\thmref{regdiscrete}).

Further examples of adjacent t-structures arise by considering certain spectral stacks.
Namely, consider the stable \inftyCat $\sC$ of quasi-coherent sheaves on the spectral quotient stack\linebreak
$[\Spec (R)/G]$, where $R$ is a noetherian $\Einfty$-ring spectrum over a field $k$ with $\pi_0(R)$ of finite type, and $G$ is a linearly reductive algebraic group over $k$ acting on $R$.
We show that $\sC$ admits a weight structure whose heart is generated by the pullbacks of representations from $BG$ (\thmref{StacksWeightStructure}).
We define a notion of \emph{homological regularity} of spectral stacks that generalizes regularity of ring spectra, and for spectral quotient stacks as above it captures precisely whether an adjacent t-structure exists on the subcategory of compact objects in $\sC$.
Under minor assumptions we prove that homological regularity of $[\Spec (R)/G]$ is equivalent to regularity
of $R$ (\thmref{regularityQuotientStack}).
In particular, if $R$ is bounded, then the spectral stack is automatically a discrete quotient stack (\corref{cor:discreteness stacks}).
The boundedness of $R$ here corresponds to a strengthening of the boundedness assumption on the adjacent t-structure on $\sC$, namely, that all objects of the heart of the weight structure on $\sC$ are $N$-truncated for some large enough $N$.

According to the Tannakian formalism developed in \cite{Iwanari}, any symmetric monoidal stable $\infty$-category $\sC^{\otimes}$ satisfying certain assumptions is symmetric monoidally equivalent to the \inftyCat of quasi-coherent sheaves on a spectral quotient stack.
This leads us to conjecture that if $\sC^{\otimes}$ is any symmetric monoidal stable \inftyCat with a compatible weight structure, then existence of a bounded adjacent t-structure, satisfying the strong boundedness condition mentioned above, implies discreteness of the heart of the weight structure (\conjref{mainth}).
In particular, this implies $\sC$ is equivalent both to the derived \inftyCat of its t-heart, and to the \inftyCat of bounded complexes over its weight-heart.
Our results on spectral stacks verify this conjecture under certain conditions on $\sC$ (\remref{rem:Iwanari}).

We end the paper by discussing the ``non-commutative'' examples of stable $\infty$-categories where an adjacent t-structure does exist.
Under appropriate necessary compatibility conditions we prove that one can glue together a weight structure from weight structures on a
semi-orthogonal decomposition, as well as an adjacent t-structure from adjacent t-structures on the semi-orthogonal decomposition. This shows in particular that triangular matrix $\Einfty$-rings are regular
simiarly to the classical triangular matrix rings.

\ssec*{Acknowledgments}
Adeel Khan participated in the development of the results and the ideas of the paper as
much as the current author, but withdrew his co-authorship for personal reasons.
We are deeply grateful to Charanya Ravi and the anonymous referee for proofreading a preliminary version of the paper and pointing out issues
in Proposition~\ref{localization} and Theorem~\ref{regularityQuotientStack}.
Furthermore, we would like to thank Benjamin Antieau, Mikhail Bondarko, Denis-Charles Cisinski, and Liran Shaul for other interesting comments.
The author was supported by the grant of the Government of the
Russian Federation for the state support of scientific research carried out under the supervision of leading scientists, agreement 14.W03.31.0030 dated 15.02.2018. The work was additionally supported by Ministry of Science and Higher Education of the Russian Federation, agreement 075–15–2019–1619.

\changelocaltocdepth{2}

\section{Regular ring spectra}
\label{sec:discretenessandregularity}

\ssec{Preliminaries}

Let $R$ be a connective \Ering, and write $\Mod_R$ for the \inftyCat of left $R$-modules.
We recall a few finiteness conditions from \cite[\sectsign~7.2.4]{HA-20170918}.

\sssec{}

The \emph{perfect} left $R$-modules are those built out of $R$ under finite (co)limits and direct summands.
These span a thick subcategory $\Mod_R^\perf \subset \Mod_R$.

\sssec{}

Suppose that $R$ is \emph{left noetherian}, i.e. that $\pi_0(R)$ is left noetherian as an ordinary ring and the homotopy groups $\pi_i(R)$ are finitely generated as left $\pi_0(R)$-modules.
A left $R$-module $M$ is \emph{almost perfect} if it is bounded below and each $\pi_i(M)$ is finitely generated as a left $\pi_0(R)$-module \cite[Prop.~7.2.4.17]{HA-20170918}, and \emph{coherent} if it is almost perfect and bounded above.
By \cite[Prop.~7.2.4.23(4)]{HA-20170918}, an almost perfect $R$-module is perfect if and only if it is of finite tor-amplitude.
We write $\Mod_R^\coh$ for the full subcategory of $\Mod_R$ spanned by coherent left $R$-modules; this is also a thick subcategory by \cite[Prop.~7.2.4.11]{HA-20170918}.

\begin{rem}
Note that the noetherian hypothesis guarantees that $R$ is almost perfect as a left $R$-module.
Thus if $R$ is bounded above, then it is moreover coherent as a left $R$-module.
It follows then that every perfect left $R$-module is coherent, i.e., that there is an inclusion $\Mod_R^\perf \subset \Mod_R^\coh$ when $R$ is bounded above.
\end{rem}

The following definition can be found in \cite{BarwickLawson}, where it is called ``almost regularity''.

\begin{defn}
Let $R$ be a left noetherian connective \Ering.
We will say that $R$ is \emph{regular} if every coherent left $R$-module is perfect.
\end{defn}

It suffices to require that every \emph{discrete} coherent left $R$-module is perfect, in view of the exact triangles
  \begin{equation*}
   \Sigma^i\pi_i(M) \to \tau_{\le i}(M) \to \tau_{\le i-1}(M)
  \end{equation*}
for $M\in\Mod_R$ and $i\in\bZ$.
Moreover, one has the following useful criterion (see \cite[Prop.~1.3]{BarwickLawson}):

\begin{prop}\label{regcriterion}
Let $R$ be a left noetherian connective $\sE_1$-ring such that $\pi_0(R)$ is a regular commutative ring.
Then $R$ is regular if and only if $\pi_0(R)$ is perfect as an $R$-module.
\end{prop}

\begin{lem}\label{lem:pi_0(R)/m criterion for regularity}
Let $R$ be a left noetherian connective $\sE_1$-ring such that $(\pi_0(R), \mfr{m})$ is a local commutative ring.
Suppose that for every discrete coherent right $R$-module $M$, the tensor product $M \otimes_R \pi_0(R)/\mfr{m}$ is bounded.
Then $R$ is regular.
\end{lem}
\begin{proof}
It suffices to show that any discrete coherent right $R$-module $M$ is of finite tor-amplitude.
By our assumption, $M \otimes \pi_0(R)/\mfr{m}$ is $k$-truncated for some $k$.
For any prime ideal $\mfr{p} \subset \pi_0(R)$ there exists a  minimal free resolution of the $\pi_0(R)/\mfr{p}$-module $\tilde{M} := M \otimes_R \pi_0(R)/\mfr{p}$ by \cite[Theorem~2.4]{Roberts} (cf.~\cite[Appendix~I, Prop.~1(a)]{Serre}).
That is, there is a quasi-isomorphism $F \to \tilde{M}$ (where we regard $\tilde{M}$ as a chain complex of discrete $\pi_0(R)/\mfr{p}_i$-modules), where $F$ is a complex of finite free $\pi_0(R)/\mfr{p}_i$-modules whose differentials are defined by matrices with coefficients in $\mfr{m}$.
The homotopy groups
  \begin{equation*}
    \pi_j \big( F \otimes_{\pi_0(R)/\mfr{p}} \pi_0(R)/\mfr{m} \big)
      = \pi_j \big( \tilde{M} \otimes_{\pi_0(R)/\mfr{p}} \pi_0(R)/\mfr{m} \big)
      = \pi_j \big( M \otimes_R \pi_0(R)/\mfr{m} \big)
  \end{equation*}
vanish for $j>k$, hence all terms of $F$ must be zero in degrees higher than $k$.
This implies that that $M \otimes_R \pi_0(R)/\mfr{p}$ is $k$-truncated.

Now any discrete coherent left $R$-module $N$ admits a filtration
$$0 = N_0 \subset \cdots \subset N_n = N $$
such that the associated graded modules $N_i/N_{i-1}$ are of the form $\pi_0(R)/\mfr{p}_i$ for some prime ideals $\mfr{p}_i$.
Considering the exact sequences
  $$ N_{i-1} \to N_i \to \pi_0(R)/\mfr{p}_i $$
and using the fact that $M \otimes_R \pi_0(R)/\mfr{p}_i$ is $k$-truncated for all $i$,
we see that $M \otimes N$ is $k$-truncated as well. Thus $M$ has finite tor-amplitude, as desired.
\end{proof}

\begin{rem}\label{rem:regularity and t-structure}
Let $R$ be a 
connective $\sE_1$-ring.
The \inftyCat $\Mod_R$ admits a canonical t-structure, where the non-negative part $(\Mod_R)_{t\ge 0}$ is spanned by the connective modules.
If $R$ is left noetherian this always restricts to the full subcategories of almost perfect and coherent modules \cite[Prop.~7.2.4.18]{HA-20170918}.
One might ask whether it also restricts to the full subcategory $\Mod_R^{\perf}$ of perfect $R$-modules.
This is the case if and only if $\pi_i(M)$ belongs to $\Mod_R^{\perf}$ for any $M \in \Mod_R^{\perf}$. When $R$ is left noetherian this is equivalent to $R$ being regular.
\end{rem}

\begin{rem}
For every $M\in \Mod_R^\perf$ and $N \in \Mod_R^\coh$, there exists an integer $n$ such that the mapping space $\Mod_R(\Sigma^n M,N)$ is contractible. Indeed, this is true for $M=\Sigma^jR$ since $N$ is bounded above, and therefore for an arbitary $M$, since it is built out of $\Sigma^jR$ using finite colimits and retracts.
For regular $R$, this says that the space $\Mod_R(M,N)$ has finitely many non-zero homotopy groups, for any $M,N \in \Mod_R^{\perf}$.
\end{rem}

\begin{exams}\label{exams:regular}
\noindent{(i)}
If $R$ is discrete, then it is regular if and only if it is regular in the sense of ordinary commutative algebra.
This follows from Auslander-Buchsbaum-Serre theorem.

\noindent{(ii)}
Let $k$ be a regular commutative ring and $R = k[T]$ with $T$ in degree 2.
Then $R$ is regular.

\noindent{(iii)}
The $\Einfty$-ring spectra $\mrm{ku}$, $\mrm{ko}$, and $\mrm{tmf}$ are regular \cite{BarwickLawson}.
\end{exams}

\sssec{}

We say that $R$ is \emph{quasi-commutative} if, for every $x \in \pi_0(R)$ and $y \in \pi_n(R)$, the equality $xy = yx$ holds in $\pi_n(R)$.
Under this assumption, we can show that regularity is stable under localizations:

\begin{prop}\label{localization}
Let $R$ be a quasi-commutative connective $\sE_1$-ring spectrum which is left noetherian.

\noindent{\em(i)}
If $R$ is regular, then for any set $S \subset \pi_0(R)$, the localization $S^{-1}R$ is also regular.

\noindent{\em(ii)}
If the localization $R_{\mfr{m}}$ is regular for every maximal ideal $\mfr{m} \subset \pi_0(R)$ and
furthermore $\pi_*(R)$ is a left noetherian graded ring, then $R$ is regular.
\end{prop}

To prove the result we need two additional lemmas.

\begin{lem}\label{lifting_maps}
Let $R$ be a quasi-commutative connective $\sE_1$-ring spectrum which is left noetherian.
Let $M$ be an $R$-module and ${\mfr{m}}$ be a maximal ideal in $\pi_0(R)$.

Then any map $P' \to M_{\mfr{m}}$ in $\Mod_{R_{\mfr{m}}}$ with $P' \in \Mod^{perf}_{R_{\mfr{m}}}$ is equivalent to
$P_{\mfr{m}} \stackrel{f_{\mfr{m}}}\to M_{\mfr{m}}$ for some  $P\in \Mod_R^{perf}$ and $f\in \Mod_R(P,M)$.
\end{lem}
\begin{proof}
The functor
$$\Mod_R \stackrel{(-)_{\mfr{m}}}\to \Mod_{R_{\mfr{m}}}$$
is a localization whose kernel $K$ is compactly generated
(see \cite[Lemma~7.2.3.13]{HA-20170918}).
Hence, this functor restricts to a fully faithful embedding
$$\Mod_R^\perf/(K\cap \Mod_R^\perf) \stackrel{L}\to \Mod_{R_{\mfr{m}}}^\perf$$
which is an equivalence up to idempotent completion.
Since $\pi_0(R)_{\mfr{m}}$ is a local ring,
$K_0(R_{\mfr{m}}) = K_0(\pi_0(R_{\mfr{m}})) = \mathbb{Z}$ and so
the functor $L$ induces an isomorphism on $K_0$.
Now by \cite[Theorem~2.1]{ThomasonSubcategories} $L$ is an
equivalence and hence there exists $P'' \in \Mod_{R_{\mfr{m}}}^\perf$ such that $P''_{\mfr{m}} \cong P'$.

Any morphism $f' \in \Mod_{R_{\mfr{m}}}(P, M_{\mfr{m}}) \cong \Mod_{R}/K(P'', M)$ can be represented by a zig-zag of morphisms in $\Mod_R$
$$P'' \stackrel{s}\leftarrow P \stackrel{f}\to M$$
where $\Cofib(s) \in K$. Since $P''$ is compact, and $K$ is compactly
generated we can assume $\Cofib(s)$ and $P$ to be compact.
This concludes the proof since by construction $f_{\mfr{m}}$ is equivalent to
$f'$ and $s$ becomes an isomorphism in $\Mod_{R_{\mfr{m}}}$.
\end{proof}

\begin{lem}\label{finite_generation_pistar}
Let $R$ be a connective $\sE_1$-ring spectrum such that $\pi_*(R)$ is left noetherian graded ring.
Then $\pi_*(M)$ is a finitely generated $\pi_*(R)$-module for any $P \in \Mod_R^{perf}$.
\end{lem}
\begin{proof}
The claim is true for $M=\Sigma^nR$ for $n\in \mathbb{Z}$.
For a fiber sequence $M_1 \to M_2 \to M$ we have
an exact sequence of graded $\pi_*(R)$-modules
$$\pi_*(M_2) \to \pi_*(M) \to \pi_{*-1}(M_1).$$
Since $\pi_*(R)$ is noetherian, this implies that $\pi_*(M)$ is finitely generated whenever
$\pi_*(M_1)$ and $\pi_*(M_2)$ are.
Moreover, the property of having finitely generated $\pi_*$ is closed under retracts, therefore it holds for all $M\in \Mod^{perf}_{R}$ by
definition of this category.
\end{proof}

\begin{proof}[proof of Proposition~\ref{localization}]
For the purposes of the proof given a subset $S \subset \pi_0(R)$ we denote the extension of scalars
functor
$\Mod_R \to \Mod_{S^{-1}R}$ by $L_S$.

\noindent{(i)}
It will suffice to show that, for every perfect $S^{-1}R$-module $N$, we have $\pi_0(N) \in \Mod^\perf_{S^{-1}R}$.
Since any perfect $S^{-1}R$-module $N$ is a retract of an object of the form $L_S(M)$ for some $M \in \Mod^\perf_R$, we may restrict our attention to $N = L_S(M)$.
For $M \in \Mod^\perf_R$, it follows from \cite[Prop.~7.2.3.20]{HA-20170918} that we have
  $$\pi_n(L_SM) \simeq S^{-1}\pi_n(M) \simeq \pi_n(M) \otimes_{\pi_0(R)} S^{-1}\pi_0(R).$$
This implies in particular that $\pi_0(L_SM) \simeq L(\pi_0(M))$ belongs to $\Mod^\perf_{S^{-1}R}$, as claimed.

\noindent{(ii)}
Let $M$ be a discrete coherent $R$-module.
By the assumption $M_{\mfr{m}}$ is perfect for any maximal ideal $\mfr{m} \subset \pi_0(R)$. Applying Lemma~\ref{lifting_maps} to the identity map on
$M_{\mfr{m}}$ we see that there is $P \in \Mod^{perf}_R$ and a map $P \stackrel{f}\to M$ that becomes an equivalence
after localizing at $\mfr{m}$.
For any $S \subset \pi_0(R)$ the map of finitely generated $\pi_*(R)$-modules
$$\pi_*(P) \stackrel{\pi_*(f)}\to \pi_*(M)$$
can be identified with
$$\pi_*(L_SP) \stackrel{\pi_*(L_Sf)}\to \pi_*(L_SM)$$
after localizing at $S$ considered as a subset of $\pi_*(R)$ by \cite[Prop.~7.2.3.20]{HA-20170918}. Therefore it becomes an isomorphism after inverting $(\pi_0(R)\setminus\mfr{m})\subset \pi_*(R)$.
By Lemma~\ref{finite_generation_pistar} $\pi_*(P)$ is a finitely generated graded $\pi_*(R)$-module. The same is true for $\pi_*(M)$, since it is a finitely generated discrete $\pi_0(R)$-module.
Hence we can find an element $x_\mfr{m} \in \pi_0(R)\setminus\mfr{m}$
such that
$$\pi_*(L_{\{x_\mfr{m}\}}P) \stackrel{\pi_*(L_{\{x_\mfr{m}\}}f)}\longrightarrow \pi_*(L_{\{x_\mfr{m}\}}M)$$
is also an isomorphism. Therefore $L_{\{x_\mfr{m}\}}M$ is perfect.
Now the ideal generated by $x_\mfr{m}$ for all $\mfr{m}$ is the unit ideal since it is not contained in
any maximal ideal.
Pick any finite set $\{x_{\mrm{m}_1},\cdots,x_{\mrm{m}_n}\}$ that also generates the unit
ideal. The family $\{R \to R_{x_{\mrm{m}_i}}\}_{i=1}^n$ is a fpqc cover, so by
\cite[Proposition~2.8.4.2(10)]{SAG-20180204} $M$ is perfect.
\end{proof}

\ssec{Bounded regular ring spectra}

This subsection is dedicated to the following result, which was previously noted for dg-algebras by
J{\o}rgensen in \cite[Theorem~A]{J_rgensen_2010} and for $\sE_\infty$-ring spectra by Lurie in
\cite[Lemma~11.3.3.3]{SAG-20180204}.
Here we only note that the same proof works more generally for quasi-commutative $\sE_1$-ring spectra.
Note that this is optimal in the sense that the statement is false without the quasi-commutativity assumption (see \ssecref{ncmcase}).

\begin{thm}\label{regdiscrete}
Let $R$ be a quasi-commutative connective $\sE_1$-ring spectrum which is left noetherian.
If $R$ is regular and has only finitely many nontrivial homotopy groups, then $R$ is discrete.
\end{thm}

\begin{proof}
We can clearly assume that $R$ is nonzero.
To show that $R$ is discrete, it will suffice to show that its localization $R_{\mathfrak{p}}$, at any prime ideal $\mathfrak{p} \subset \pi_0(R)$, is discrete.
Since $R_{\mathfrak{p}}$ is regular by Proposition~\ref{localization}, we may replace $R$ by $R_{\mathfrak{p}}$ and thereby assume that $\pi_0(R)$ is local.
Denote by $\mfr{m} \subset \pi_0(R)$ the maximal ideal, and $\kappa$ the residue field.
Let $n \ge 0$ be the largest integer such that $\pi_{n}(R) \ne 0$.
For the sake of contradiction, we will assume that $n>0$.

Since $R$ is noetherian, the $\pi_0(R)$-module $\pi_n(R)$ is finitely generated.
Let $\mfr{p} \subset \pi_0(R)$ be a minimal prime ideal such that $\mfr{p} \in \on{Supp}_{\pi_0(R)}(\pi_n(R))$.
By Proposition~\ref{localization}, the localization $R_{\mathfrak{p}}$ is regular.
Therefore, replacing $R$ by $R_{\mathfrak{p}}$, we may assume that $\pi_n(R)$ is supported at $\mfr{m}$, so that in particular $\mfr{m} \subset \pi_0(R)$ is an associated prime ideal of $\pi_n(R)$.

Since $R$ is regular, the coherent left $R$-module $\kappa$ is perfect.
Therefore, the following claim will yield the desired contradiction:
\begin{enumerate}
  \item[$(\ast)$]
Let $N$ be a nonzero connective perfect left $R$-module.
Let $k\ge 0$ be the smallest natural number such that $N$ is of tor-amplitude $\le k$.
Then we have $\pi_{n+k}(N) \ne 0$.
\end{enumerate}

To prove $(\ast)$ we argue by induction.
In the case $k=0$, $N$ is flat with $\pi_0(N) \ne 0$, so that $\pi_{n}(N) \simeq \pi_{n}(R) \otimes_{\pi_0(R)} \pi_0(N) \ne 0$.
Assume therefore that $k>0$.
Choose elements of $\pi_0(N)$ that induce a basis of the $\kappa$-vector space $\pi_0(N) \otimes_{\pi_0(R)} \kappa$ and consider the surjective $R$-module morphism $u : R^{\oplus m} \to N$ they determine.
Its fiber, which we denote $N'$, is a connective perfect left $R$-module fitting in an exact triangle
  \begin{equation*}
     N' \xrightarrow{v} R^{\oplus m} \xrightarrow{u} N.
   \end{equation*}
Considering the induced exact sequence of abelian groups
  \begin{equation*}
    0 = \pi_{n+k}(R^{\oplus m}) \to \pi_{n+k}(N) \to \pi_{n+k-1}(N') \xrightarrow{\varphi} \pi_{n+k-1}(R^{\oplus m}),
  \end{equation*}
it will suffice to show that $\varphi$ is not injective.

Since $N$ is of tor-amplitude $\le k$, with $k$ minimal and $> 0$, it follows that $N'$ is nonzero and of tor-amplitude $\le k-1$.
Therefore by the inductive hypothesis, $\pi_{n+k-1}(N') \ne 0$.
If $k>1$, then $\pi_{n+k-1}(R^{\oplus m}) = 0$, so $\varphi$ is not injective.
It remains to consider the case $k=1$.
In this case, the perfect left $R$-module $N'$ is flat, and hence free of finite rank $r\ge 0$, since $\pi_0(R)$ is local.
Therefore the map $v : N' \to R^{\oplus m}$ is determined up to homotopy by a matrix $\{v_{i,j}\}_{1\le i\le r,1\le j\le m}$ of elements $v_{i,j} \in \pi_0(R)$.
Since $\mfr{m} \subset \pi_0(R)$ is an associated prime of $\pi_n(R)$, we may choose a nonzero element $x \in \pi_n(R)$ such that multiplication by $x$ kills $\mfr{m}$.
We claim that the element $(x,x,\ldots,x) \in \pi_n(R)^{\oplus r} \simeq \pi_n(N')$ is sent by $\varphi$ to zero in $\pi_n(R^{\oplus m})$.
For this purpose it will suffice to show that the elements $v_{i,j} \in \pi_0(R)$ all belong to $\mfr{m}$.
Consider the exact sequence
  \begin{equation*}
    \pi_0(N')/\mfr{m} \xrightarrow{v} \pi_0(R^{\oplus m})/\mfr{m} \xrightarrow{u} \pi_0(N)/\mfr{m} \to 0.
  \end{equation*}
Since the map $u : \pi_0(R^{\oplus m})/\mfr{m} \to \pi_0(N)/\mfr{m}$ is injective, it follows that $v : \pi_0(N')/\mfr{m} \to \pi_0(R^{\oplus m})/\mfr{m}$ is the zero map.
In other words, the image of $v : \pi_0(N') \to \pi_0(R^{\oplus m})$ lands in  $\mfr{m}^{\oplus m} \subset \pi_0(R)^{\oplus m}$, as desired.
\end{proof}

\ssec{A noncommutative counterexample}\label{ncmcase}

We now give an example of a \emph{noncommutative} regular ring spectrum that is bounded but not discrete.
This shows that the quasi-commutativity hypothesis in \thmref{regdiscrete} is necessary.

\begin{constr}\label{constr:noncommutative counterexample}
Let $k$ be a field.
Consider the graded ring $R$, concentrated in degrees $0$ and $1$, where
$R_0 = k \times k$ and $R_1 = k$
with multiplication defined by the formulas
$$(a,b)\cdot (c,d)= (ac,bd)$$
$$x\cdot (a,b)= bx$$
$$(a,b)\cdot x = ax,$$
where $(a,b),(c,d) \in R_0$, $x \in R_1$.
\end{constr}

The graded ring $R$ is associative but not commutative, and gives rise to a connective $\sE_1$-ring spectrum, which we denote again by $R$, that is not quasi-commutative.
However, it has regular $\pi_0$, and is also itself regular:

\begin{prop}
Let $R$ be the $\sE_1$-ring spectrum defined by \constrref{constr:noncommutative counterexample}.
Then $R$ is regular.
\end{prop}
\begin{proof}
By \propref{regcriterion} it suffices to show that $\pi_0(R)$ is perfect.

Denote by $X$ the direct summand of the free $R$-module $R$ corresponding to the projector $(0,1) \in \pi_0(R) = k \times k$.
We first note that $X \simeq k$ is discrete.
Indeed we have by definition $\pi_0(X) = k$ and $\pi_i(X) = 0$ for $i\neq 0,1$.
The only $x \in k$ satisfying $(0,1)\cdot x = x$ is zero, so $\pi_1(X)$ is also zero.

Since $x\cdot (0,1) = x \in \pi_1(R)$, the element $1 \in k \simeq \pi_1(R)$ gives rise to a map $X[1] \xrightarrow{1} R$.
This map induces an isomorphism $\pi_0(X) \to \pi_1(R)$, and therefore exhibits $X[1]$ as the $1$-connective cover $\tau_{\ge 1} R$.
It follows that its cofiber is an object of $\Mod_R^\perf$ equivalent to $\tau_{\le 0} R = \pi_0(R)$.
\end{proof}

\begin{rem}
In fact it is easy to see using the Künneth spectral sequence that any $\sE_1$-ring spectrum $R$ is regular as soon as the classical ring $\pi_*(R) = \bigoplus_n \pi_n(R)$ is regular.
In the example above, the ring $\pi_*(R)$ is isomorphic to the ring of triangular matrices, which is regular.
\end{rem}

\ssec{Regular spectral stacks}
In this section we study regularity for spectral stacks.
In particular we show an analogue of \thmref{regdiscrete} for quotient stacks.

Recall that a functor $F \colon \Einfty-\operatorname{Rings} \to \Spt$ is called a {\it spectral stack} if the following holds:
\begin{itemize}
  \item $F$ is a sheaf in the fpqc topology on $\Einfty-\operatorname{Rings}^{op}$.
  \item there exists a surjection of sheaves $\Spec (R) \stackrel{f}\to F$, where $\Spec$ denotes
  the Yoneda embedding $\Einfty-\operatorname{Rings}^{op} \to \Fun(\Einfty-\operatorname{Rings}, \Spt)$.
  We call this map an {\it atlas} of $F$.
  \item the map $\Spec(R) \to F$ of the map $f$ is flat, that is for any map $\Spec(S) \to F$ the pullback map
  $$\Spec(S) \times_F \Spec(R) \to \Spec(R)$$
  is equivalent to $\Spec(f_R)$ for some flat map of $\sE_\infty$-rings $R \stackrel{f_R}\to R'$.
\end{itemize}
We say that a spectral stack is connective if the ring spectrum $R$ can be chosen to be connective.

Note that any $\Einfty$-ring can be viewed as a spectral stack via taking the representable functor.
Denote by $\Stk_{\infty}$ the $\infty$-category of spectral stacks.
We define $\Qcoh \colon \Stk_{\infty} \to \on{Cat}_{\infty}^{ex}$ to be the right Kan extension of the
functor $\Mod_R \colon \Einfty-\operatorname{Rings} \to \on{Cat}_{\infty}^{ex}$. This functor  admits a
symmetric monoidal enhancement (see \cite[6.2.6]{SAG-20180204}). Moreover $\Qcoh(X)$ admits a t-structure
where the nonnegative part is the full subcategory of connective objects, i.e., those objects that become
connective after pullback to a flat atlas (see \cite[Corollary~9.1.3.2]{SAG-20180204}).

\sssec{}
An affine group scheme $G$ is said to be {\it linearly reductive} if the functor of taking invariants
$$(-)^G\colon \Qcoh(BG) \to \Qcoh_{\Spec k}$$
is t-exact. We say that a group scheme $G$ acts on
$\Einfty$-$k$-algebra $R$ if $R$  an $\Einfty$-algebra object in $\Qcoh(BG)$.
Given an affine group scheme $G$ over a field $k$ acting on a connective $\Einfty$-$k$-algebra $R$ consider the functor
$$[\Spec R/G] \colon \Einfty-\operatorname{Rings} \to \Spt$$
which is the colimit of the simplicial
diagram
$$
\begin{tikzcd}
\cdots G \times_k G \times_k \Spec R
 \arrow[r,shift left, shift left]\arrow[r]\arrow[r, shift right, shift right]& G \times_k \Spec R \arrow[l,shift left]\arrow[l,shift right]\arrow[r, shift left]\arrow[r,shift right] &\Spec R\arrow[l]
\end{tikzcd}
$$
in $\Fun(\Einfty-\operatorname{Rings}, \Spt)$.
This turns out to be a stack and an atlas is given by $\Spec R \to [\Spec R/G]$ (see \cite[Corollary 9.1.1.5]{SAG-20180204}). We call it the {\it quotient stack}. When $R=k$ and the action is trivial we adopt the notation $BG := [\Spec k/G]$.
We have
$$\Qcoh([\Spec R/G]) \simeq \lim \Qcoh(G^{\times n} \times_k \Spec R) \simeq \lim \Mod_{\mathcal{O}_G^{\otimes n} \otimes_k R}(\Qcoh(G^{\times n})) \simeq \Mod_R(\Qcoh(BG)),$$
where $R$ is considered as an $\sE_{\infty}$-algebra object in $\Qcoh(BG)$ on the right-hand side.
(cf. \cite[Proposition~2.7]{Iwanari}). Alternatively, we may identify the \inftyCat $\Qcoh([\Spec R/G])$ with $R$-modules endowed with an action of $G$ which is compatible with an action on $R$.

We refer to \cite[Sect.~9]{SAG-20180204}
for more generalities on spectral stacks.

\begin{defn}\label{regularity:stacks}\leavevmode

\noindent{(i)}
We say that $X \in \Stk_\infty$ is \emph{homologically regular} if the t-structure on $\Qcoh(X)$ restricts to the subcategory of compact objects.

\noindent{(ii)}
We say that $X\in \Stk_\infty$ is \emph{regular} if there exists a smooth, faithfully flat morphism $\Spec(R) \to X$ with $R$ a regular connective $\Einfty$-ring spectrum.
\end{defn}

The two definitions are not in general equivalent because homological regularity is not local (even in the étale topology).
Indeed, for a field $k$ of positive characteristic $p$ and a finite discrete group $G$ of order $p$, the quotient stack $[\Spec(k)/G]$ is not homologically regular.

However, for $X=\Spec (R)$, where $R$ is some connective $\Einfty$-ring spectrum, the two notions are equivalent.
Indeed, (i) implies (ii) tautologically, while the converse implication follows from \cite[Proposition~2.8.4.2]{SAG-20180204}.

\sssec{}

For the rest of the section, we fix a field $k$, $G$
an affine group scheme, and $R$ a noetherian connective $\sE_{\infty}$-$k$-algebra with an action of $G$ such that $\pi_0(R)$ is a finite type $k$-algebra. We additionally
assume that either $char(k)=0$ or $G$ is linearly reductive.
We write $X = [\Spec(R)/G]$ and $f \colon X \to BG$, $p \colon \Spec R \to X$ for the canonical morphisms.
Under the assumptions the functor $f^* \colon \Qcoh(BG) \to \Qcoh(X)\simeq \Mod_R(\Qcoh(BG))$ can be identified with $V \mapsto V \otimes R$.
The inverse image functor
  \begin{equation*}
    p^* \colon \Qcoh(X) \to \Mod_R,
  \end{equation*}
is identified with the functor of forgetting the action.

The main theorem of this section is:

\begin{thm}\label{regularityQuotientStack}
The spectral stack $[\Spec(R)/G]$ is homologically regular if $R$ is regular.
Moreover, the converse implication holds if $G$ is smooth and $\pi_*(R)$ is graded noetherian.
\end{thm}

In particular, we see that for the spectral stack $[\Spec(R)/G]$ with smooth $G$ and noetherian $\pi_*(R)$, homological regularity coincides with regularity.
Note that the graded ring $\pi_*(R)$ is automatically noetherian if $R$ is a bounded and noetherian
$\sE_{\infty}$-ring spectrum.
Combining this with \thmref{regdiscrete}, we get:

\begin{cor}\label{cor:discreteness stacks}
Assume that $G$ is smooth.
If $R$ is bounded and the spectral stack $[\Spec R/G]$ is homologically regular, then $R$ is discrete.
\end{cor}

\begin{prop}\label{prop:compact generation of Mod_X}
The stable \inftyCat $\Qcoh(X)$ is compactly generated by the inverse images $f^*(V) \in \Qcoh(X)$, for all irreducible $k$-representations $V \in \Qcoh(BG)$.
\end{prop}

\begin{proof}
By \cite[Thm.~A(c)]{hallandrydh}, the \inftyCat $\Qcoh(BG)$ is compactly generated by the irreducible $k$-representations of $G$.
Thus the claim follows from \lemref{modulesGeneration} below.
\end{proof}


\begin{lem}\label{modulesGeneration}
Let $\sC^{\otimes}$ be a stable symmetric monoidal $\infty$-category whose underlying $\infty$-category is compactly generated by $\{X_i\}$ and the
unit is compact.
Then for any $\sE_1$-algebra object $A$ in $\sC^{\otimes}$ the \inftyCat of left $A$-modules $\Mod_A(\sC)$ is compactly generated by $\{A\otimes X_i\}$.
\end{lem}

\begin{proof}
Let $U : \Mod_A(\sC) \to \sC$ be the forgetful functor, right adjoint to $Y \mapsto A \otimes Y$.
By Corollaries~4.2.3.7(2) and 4.2.3.2 of \cite{HA-20170918}, $U$ preserves colimits and is conservative.
This implies that the objects $A \otimes X_i$ are compact.
It remains to show that they form a set of generators.
If $Y \in \sC$ is an object such that the mapping space $\Maps_A(A\otimes X_i, Y[j]) = \Maps_\sC(X_i,U(Y)[j])$ is contractible for all $i$ and all $j\in\bZ$, then $U(Y)=0$.
This implies that $Y=0$, whence the claim.
\end{proof}

\begin{proof}[Proof of \thmref{regularityQuotientStack}]
First, assume that $R$ is regular.
To show $X$ is homologically regular it suffices to show that for each of the compact generators $f^*(V)$ from \propref{prop:compact
generation of Mod_X}, $\pi_n(f^*(V))$ is compact for all $n$. By t-exactness of $p^*$ (see \cite[Remark~9.1.3.4]{SAG-20180204}), we have
$p^*\pi_n(f^*(V)) = \pi_n(p^*f^*(V))$.
Since $p^*f^*(V) = R\otimes_k V$ is a free $R$-module, it is a compact object of $\Mod_R$.
If $R$ is regular, then $\pi_n(p^*f^*(V))$ is compact, so it follows from
\cite[Prop.~9.1.5.3]{SAG-20180204} and compactness of $\sO_X = f^*(k)$ that $\pi_n(f^*(V))$ is also compact.
Thus $X$ is homologically regular.

Now we prove the 'moreover' part. Assume $R$ is not regular.
We will argue by contradiction that $X$ cannot be homologically regular.
By Proposition~\ref{localization} we know that $R_{\mfr{m}}$ is not regular for some maximal ideal $\mfr{m} \subset \pi_0(R)$.
By \lemref{lem:pi_0(R)/m criterion for regularity} there exists a discrete coherent $R$-module $M$ such that $M_{\mfr{m}} \otimes_{R_{\mfr{m}}} \pi_0(R_{\mfr{m}})/\mfr{m}$ has infinitely many nonzero homotopy groups.
It follows that $M \otimes_R \pi_0(R)/\mfr{m}$ is also unbounded.
By \cite[Prop.~7.4]{Milne} and the smoothness assumption the $G$-orbit of the point of $\Spec (\pi_0(R))$ corresponding to the ideal $\mfr{m}$ is regular.
Therefore, if $I \subset \pi_0(R)$ denotes the ideal corresponding to the closure of the orbit, then there is an element $f \not\in \mfr{m}$ such that $(\pi_0(R)/I)_f$ is regular.
Now $\pi_0(R)/\mfr{m}$ has a structure of a $(\pi_0(R)/I)_f$-module and we have an equivalence
$$M \otimes_R \pi_0(R)/\mfr{m} \cong \big( M\otimes_R(\pi_0(R)/I)_f \big)  \otimes_{(\pi_0(R)/I)_f} \pi_0(R)/\mfr{m}.$$
In particular, the right-hand side is unbounded.
But since $(\pi_0(R)/I)_f$ is regular, this implies that the $(\pi_0(R)/I)_f$-module
  $$ M\otimes_R(\pi_0(R)/I)_f \simeq (M\otimes_R \pi_0(R)/I)_f $$
is unbounded and hence so is $M\otimes_R \pi_0(R)/I$.
In particular, $\pi_0(R)/I$ cannot be perfect as an $R$-module.

Since the ideal $I$ is $G$-invariant, $\pi_0(R)/I$ comes as $p^*(M)$ for some  $M \in \Qcoh(X)$.
It follows from Proposition \ref{prop:compact generation of Mod_X} that $f^* V$ is a compact object of $\Qcoh(X)$ for any finite dimensional $k$-representation $V$ of $G$.
If $X$ is homologically regular, $\pi_0(f^*(V))$ is then also a compact object in $\Qcoh(X)$.
The same holds for $M$, since it can be written as $\pi_0$ of the homotopy cofibre of a $G$-equivariant $R$-linear map $\varphi : V\otimes_k \pi_0(R) \to \pi_0(R)$ ($V$ is any finite-dimensional subrepresentation of $\pi_0(R)$ containing a finite set of generators of $I$).
But then the $R$-module $\pi_0(R)/I$ is compact (= perfect), since the functor $p^*$ preserves compact objects, whence the desired contradiction.
\end{proof}

\begin{rem}
\thmref{regularityQuotientStack} gives a lot of examples of homologically regular spectral stacks.
In fact, its proof shows that a spectral quotient stack $X=[\Spec R/G]$, defined over any 
ring $k$, is homologically regular, provided $R$ is a regular
$\Einfty$-ring and $G$ is a finite group whose order is invertible in $k$.

Thus we see that $[\Spec(\on{ku}[1/2])/C_2]$ is a homologically regular spectral stack for the canonical action of $C_2$.
Arguing as in \cite{BarwickLawson} one can show that there is a fiber sequence
$$\K^{C_2}(\bZ[1/2]) \to \K^{C_2}(\on{ku}[1/2]) \to \K^{C_2}(\on{KU}[1/2]).$$
\end{rem}



\section{Weight structures}
\label{sec:weights}

In this section we recall the notion of \emph{adjacent structures} and show its relation to regularity.
An adjacent structure consists of a pair of a t-structure and a weight structure compatible in a strict sense.
We begin by recalling what a weight structure is.

\ssec{Definitions and basic properties}

\begin{defn}
A {\it weight structure} $w$ on a stable $\infty$-category $\sC$ is the data of two full subcategories $\sC_{w\ge 0}, \sC_{w\le 0}$ closed under retracts satisfying the following axioms:
\begin{defnlist}
  \item {\em Semi-invariance with respect to translations.}
We have the inclusions
  $$\Sigma\sC_{w\ge 0} \subset \sC_{w\ge 0},
  \quad\Omega\sC_{w\le 0} \subset \sC_{w\le 0}.$$

  \item {\em Orthogonality.}
For any $X \in \sC_{w\le 0}$, $Y\in  \sC_{w\ge 0}$, we have
  $$\pi_0\Maps_\sC(X,\Sigma Y) = 0.$$

  \item {\em Weight decompositions}.
For any $X \in \sC$, there exists an exact triangle
  $$w_{\le 0} X \to X \to w_{\ge 1} X,$$
where $w_{\le 0} X \in \sC_{w\le 0}$ and $w_{\ge 1} X \in \Sigma\sC_{\ge 0}$.
\end{defnlist}

Morally one may think of the objects of $\sC_{w\ge 0}$ and $\sC_{w\le 0}$ as of those built out of non-negative and non-positive cells, respectively.
Then the orthogonality condition corresponds to the vanishing of negative homotopy groups and the weight decompositions are the skeletal filtrations.

\sssec{}
The \emph{heart} $\hw$ of a weight structure is the full subcategory whose objects belong to both $\sC_{w\ge 0}$ and $\sC_{w\le 0}$.
A weight structure is \emph{bounded} if any object $X$ is an object of $\Sigma^n\sC_{w\le 0}$ and of $\Omega^n\sC_{w\ge 0}$ for some $n$.
A functor between stable $\infty$-categories with weight structures is \emph{weight-exact} if it preserves both classes.

\sssec{}We also adopt the following notation. For any $X \in \sC$
and a weight decomposition of $\Omega^n X$
$$w_{\le 0} (\Omega^nX) \to \Omega^nX \to w_{\ge 1} (\Omega^nX),$$
the $n$-th suspension of the triangle will be denoted by
$w_{\le n} X \to X \to w_{\ge n+1} X$
and will be also called a weight decomposition of $X$.
\end{defn}

The following construction gives rise to a variety of examples of weight structures:

\begin{prop}\label{p1}
Let $\sA$ be an additive $\infty$-category.
Denote by $\widehat{\sA}$ the stable \inftyCat $\Funadd(\sA^\op, \Spt)$ of additive presheaves of spectra on $\sA$.
Then the full subcategory $\widehat{\sA}^c$ of compact objects admits a bounded weight structure $w$ with the following properties:
\begin{thmlist}
  \item
The full subcategory $(\widehat{\sA}^c)_{w\ge 0}$ is spanned by presheaves with values in connective spectra.

  \item
The Yoneda embedding $\sA \to \widehat{\sA}$ factors through a canonical functor $\sA \to \hw$ which exhibits the heart $\hw$ as the idempotent completion of $\sA$.
\end{thmlist}
\end{prop}

In fact, this is essentially the only example of a bounded weight structure, as the following proposition shows.

\begin{prop}\label{p2}
Let $\sC$ be an idempotent complete stable $\infty$-category endowed with a bounded weight structure.
Then there is a weight-exact equivalence of $\infty$-categories $\sC \to \widehat{\hw}^c$, where $\widehat{\hw}$ denotes the construction of \propref{p1}.
\end{prop}

Both Propositions~\ref{p1} and \ref{p2} follow from \cite[Cor.~3.4]{Vovastheoremoftheheart}.
The following example will be of special interest for us:

\begin{exam}\label{exam:weight structure on Mod_R^perf}
Let $R$ be a connective $\sE_1$-ring spectrum and let $\sA$ be the additive \inftyCat $\Mod_R^\free$ of finitely generated free $R$-modules (concentrated in degree 0).
Then the construction $\widehat{\sA}^c$ of \propref{p1} is canonically equivalent to the stable \inftyCat $\Mod_R^\perf$.
It follows that $\Mod_R^\perf$ admits a canonical bounded weight structure $w$ such that the non-negative part $(\Mod_R^\perf)_{w\ge 0}$ is spanned by the connective perfect $R$-modules.
\end{exam}

\ssec{Weight structures compatible with the symmetric monoidal structure}

\begin{defn} \label{defn:weight monoidal compatible}
Let $\sC^{\otimes}$ be a stable symmetric monoidal  $\infty$-category such that the underlying $\infty$-category is endowed with a weight
structure $w$. We say that the weight structure is  {\it compatible} with the symmetric monoidal structure if
$\sC_{w\ge 0}$ and $\sC_{w\le 0}$ are closed under the tensor product operations.

In this situation the symmetric monoidal structure restricts to the subcategory $\hw$.
We denote the corresponding symmetric monoidal \inftyCat by $\hw^{\otimes}$.
\end{defn}

\sssec{}
We already know from Propositions \ref{p1} and \ref{p2} that stable $\infty$-categories together with a bounded weight structure correspond to
additive $\infty$-categories, at least up to idempotent completion.
It was shown in \cite[Lemma~4.2]{koaoki} that compatible weight structures correspond to symmetric monoidal additive $\infty$-categories in the same fashion:

\begin{prop}\label{symmonp2}
Let $\sC^{\otimes}$ be a stable idempotent complete symmetric monoidal $\infty$-category whose underlying $\infty$-category is endowed with
a compatible bounded weight structure.
Then there is a weight-exact equivalence of symmetric monoidal categories $\sC^{\otimes} \cong \widehat{\hw^{\otimes}}^c$, where
$\widehat{\hw^{\otimes}}$ is
a natural symmetric monoidal refinement of the construction $\widehat{\hw}$ from \propref{p1}.
\end{prop}

\begin{exam}\label{exam:compatible weight structure on Mod_R^perf}
Let $R$ be a connective $\sE_{\infty}$-ring spectrum and let $\sA$ be the additive \inftyCat $\Mod_R^\free$ of finitely generated free $R$-modules. The weight structure from Example~\ref{exam:weight structure on Mod_R^perf} is compatible with the canonical symmetric monoidal structure.
\end{exam}

\ssec{Adjacent structures}

The definitions of weight structure and t-structure appear formally quite similar at first glance, but they behave completely differently. For instance,
there is no simple characterization of t-structures analogous to \propref{p2}.
An actual mathematical relation between the two notions is given by the following result of Bondarko \cite[Thm.~0.1]{BondarkoAdjacentStructures}.

\begin{thm}\label{adjtstr}
Let $\sC$ be a stable $\infty$-category that admits all small coproducts and satisfies the Brown representability theorem (i.e., every coproduct preserving homological functor $\sC^{op} \to Ab$ is representable). For instance, $\sC$ is compactly generated.
Suppose $(\sC_{w\ge 0}, \sC_{w\le 0})$ is a weight structure such that $\sC_{w\ge 0}$ is closed under arbitrary coproducts.
Then there exists a t-structure on $\sC$ such that $\sC_{t\ge 0} = \sC_{w\ge 0}$.
\end{thm}

When $\sC$ does not have all coproducts, such a t-structure may not exist.
When it does, it is called an adjacent t-structure:

\begin{defn}
Let $\sC$ be a stable $\infty$-category.
Given a weight structure $(\sC_{w\ge 0}, \sC_{w\le 0})$ and a t-structure $(\sC_{t\ge 0}, \sC_{t\le 0})$, we say that they are {\it adjacent} if the equality $\sC_{t\ge 0} = \sC_{w\ge 0}$ holds.
\end{defn}

We write $\htt$ for the heart of the t-structure.

\begin{rem}
Let $\sC$ be a stable \inftyCat with a weight structure.
If an adjacent t-structure exists, then it is completely determined as follows: $\sC_{t\ge 0} = \sC_{w\ge 0}$, and $\sC_{t\le 0}$ is the full subcategory of objects $X \in \sC$ such that $\pi_0\Maps_\sC(\Sigma Y, X) = 0$ for all $Y \in \sC_{w\ge 0}$.
In other words, the existence of an adjacent t-structure is a property (as opposed to additional structure).
\end{rem}

\begin{exam}\label{exam:perfect adjacent}
Let $R$ be a regular $\sE_1$-ring spectrum and consider the \inftyCat $\Mod_R^{\perf}$.
Its canonical t- and weight structures (\remref{rem:regularity and t-structure} and \examref{exam:weight structure on Mod_R^perf}) are adjacent.
Conversely, any non-regular $R$ provides an example of a weight structure with no adjacent t-structure.
\end{exam}

\sssec{}

Let $\sC$ be a monoidal stable \inftyCat with a bounded weight structure.
We think of $\sC$ as a ring spectrum with ``many objects''.
In view of \examref{exam:perfect adjacent}, the existence of an adjacent t-structure can be viewed as a regularity condition on $\sC$.
If $\sC$ is moreover symmetric monoidal, in a way compatible with the weight structure (in the sense of \defref{defn:weight monoidal compatible}), then we think of $\sC$ as a \emph{commutative} regular ring spectrum with many objects.


With this in mind we expect the following generalization of \thmref{regdiscrete}:

\begin{conj}\label{mainth}
Let $\sC^{\otimes}$ be a stable idempotent complete symmetric monoidal $\infty$-category with
a compatible bounded weight structure $w$ (in the sense of \defref{defn:weight monoidal compatible}).
Suppose there exists a natural number $N$ such that
$\pi_n\hw(X,Y) = 0$ for all $n \ge N$ and $X, Y \in \hw$.
If $w$ admits an adjacent t-structure, then $\hw$ is discrete.
\end{conj}

In particular the conjecture implies that under the conditions, $\sC$ is equivalent to the \inftyCat of bounded complexes over $\hw$ and also to the
derived \inftyCat of $\htt$.

In fact, we will see that this conjecture not only generalizes \thmref{regdiscrete} but also \corref{cor:discreteness stacks}.

\ssec{Weight structures on stacks}

\begin{lem}\label{modulesWeight}
Let $\sC^{\otimes}$ be a stable symmetric monoidal category whose underlying $\infty$-category is compactly generated by $\{X_i\}$ and the
unit is compact.
Assume also that $\sC$ admits a compatible weight structure and the compact generators belong to the heart of the
weight structure.
Then for any weight-positive $\sE_1$-algebra object $A$ in $\sC^{\otimes}$, the \inftyCat $\Mod_A(\sC)$ admits a weight structure such that the following hold:
\begin{thmlist}
  \item
The full subcategory $\Mod_A(\sC)_{w\ge 0}$ is the closure of the set $\{A\otimes X_i\}$ under small colimits. More explicitly, it is spanned by the objects $X \in \sC$ such
that $\pi_0\Maps_A(A\otimes X_i, \Sigma^n X) = 0$ for all $i$ and all $n>0$.
  \item
The weight structure restricts to the subcategory of compact objects.
\end{thmlist}
\end{lem}
\begin{proof}
By Lemma~\ref{modulesGeneration} the set $\{A\otimes X_i\}$ compactly generates
$\sC$.
Applying \cite[Theorem~4.5.2(I.2)]{BondarkoWeights}, it suffices to show that it is negative in the sense of \cite[Definition~4.3.1]{BondarkoWeights}, i.e., that $\pi_0\Maps_A(A\otimes X_i, \Sigma^n A\otimes X_j) = 0$ for all $i,j$ and $n>0$.
By adjunction we see that
$$\pi_0\Maps_A(A\otimes X_i, \Sigma^n A\otimes X_j) = \pi_0\Maps_\sC(X_i, \Sigma^n A\otimes X_j).$$
The latter group is trivial since $X_i \in \hw$ and $\Sigma^n A\otimes X_j \in \sC_{w\ge 1}$.
The first assertion follows by construction, and the second follows from \cite[Theorem~4.3.2(II.2)]{BondarkoWeights}.
\end{proof}

\sssec{}

Let $k$ be a field and $G$ a linearly reductive group scheme of finite type over $k$.
Let $R$ be a noetherian connective $\sE_{\infty}$-ring spectrum over $k$ with a $G$-action, such that $\pi_0(R)$ is a finite type $k$-algebra.
We denote by $X = [\Spec(R)/G]$ the corresponding quotient stack and by $p : \Spec(R) \to X$ the canonical morphism.

\begin{thm}\label{StacksWeightStructure}
The \inftyCat $\Qcoh(X)$ admits a unique weight structure where $\Qcoh(X)_{w\ge 0}$ is spanned by objects $M \in \Qcoh(X)$ whose underlying $R$-module $p^*(M)$ is connective.
Moreover, the weight structure restricts to the subcategory of compact objects.
\end{thm}
\begin{proof}
By \cite[Theorem~A(c)]{hallandrydh} the \inftyCat $\Qcoh(BG)$ admits a set of compact generators given by irreducible $G$-representations.
Since $G$ is linearly reductive the category of representations of $G$ is
semisimple.
In particular, $\pi_0\Maps_{\Qcoh(BG)}(P_1,\Sigma^iP_2) = \on{Ext}_G^i(P_1,P_2) = 0$ for all $i>0$ and all irreducible representations $P_1,P_2$.
Hence, the set of irreducible representations is negative in the sense of \cite[Definition 4.3.1]{BondarkoWeights} and, by
\cite[Theorem 4.5.2(I.2)]{BondarkoWeights}, gives rise to a weight structure on $\Qcoh(BG)$ such that all irreducible representations belong to
its heart.
It is clearly compatible with the symmetric monoidal structure.

The $G$-equivariant ring spectrum $R$, considered as an object of
$\Qcoh(BG)$, is weight-positive.
Therefore by Lemma~\ref{modulesWeight} we obtain a weight structure on $\Mod_R(\Qcoh(BG))$ that restricts to the subcategory of compact objects.
Under the equivalence $\Qcoh(X) \simeq \Mod_R(\Qcoh(BG))$ (see \cite[Proposition~2.7]{Iwanari}), this induces the desired weight structure on $\Qcoh(BG)$.

It remains to show that $\Qcoh(X)_{\ge0}$ has the claimed description.
Any object of $\Qcoh(X)_{w\ge 0}$ is a colimit of objects of the form
$R\otimes V_i$, so is connective. Conversely, for any connective object
$M \in \Qcoh(X)$, any $n>0$, and any $i$, we have
  $$\pi_0\Maps_{\Qcoh(X)}(R\otimes V_i, \Sigma^n M) = \pi_0\Maps_{\Qcoh(X)}(V_i, \Sigma^n M) = H^{-n}(G, \uHom_k(V_i, M)) = 0,$$
so $M$ belongs to $\Qcoh(X)_{w\ge 0}$.
\end{proof}

\begin{rem}\label{rem:weight structure on Mod_X}
Since $\Qcoh(X)_{w\ge 0}$ is spanned by the connective objects, the weight structure is adjacent to the canonical t-structure.
The induced weight structure on the subcategory of compact objects admits an adjacent t-structure if and only if the canonical t-structure restricts to the subcategory of compact objects, in other words, if and only if $X$ is homologically regular (see Definition~\ref{regularity:stacks}).
Moreover, $R$ is bounded if and only if there exists an $N>0$ such that
$\pi_n\Maps_{\Qcoh(X)}(P,Q) = 0$ for all $n\ge N$ and all $P,Q$ from the heart of the weight structure.
Indeed, if $R$ is bounded then for all irreducible $G$-representations $V_i,V_j$ we see that
  $$ \Maps_{\Qcoh(X)}(R \otimes V_i, R\otimes V_j) = \Maps_{\Qcoh(BG)}(V_i, R\otimes V_j) = \uHom_{k}(V_i, R\otimes_k V_j)^G $$
is bounded, since the functor of $G$-invariants is exact.
Since $R \otimes V_i$ generate the heart of the weight structure, the same holds for $\Maps_{\Qcoh(X)}(P,Q)$ with $P,Q$ arbitrary objects in the heart of the weight structure.
Conversely, suppose there is an integer $N$ such that for all $n\ge N$, $\pi_n \Maps_{\Qcoh(X)}(R \otimes V_i, R\otimes k) = 0$ for all irreducible representations $V_i$, where $k$ is the trivial representation.
Then by adjunction it follows that $R \in \Qcoh(BG)$ is bounded (with respect to the canonical t-structure).
In particular, its underlying $k$-module is bounded.
\end{rem}

By \remref{rem:weight structure on Mod_X}, the conditions of \conjref{mainth} are satisfied for $\Qcoh(X)$.
In this case, the conjecture is verified by \corref{cor:discreteness stacks}.
That is, we have:

\begin{exam}\label{exam:regular stacks}
\conjref{mainth} holds for $\sC$ under the following additional assumption:
\begin{thmlist}
  \item[$\mathrm{(*)}$] There is a symmetric monoidal weight-exact equivalence of stable \inftyCats
    \begin{equation*}
      \sC \simeq \Qcoh([\Spec(R)/G]),
    \end{equation*}
  where $G$ is a linearly reductive group scheme of finite type over a field $k$, and $R$ is a noetherian $\Einfty$-$k$-algebra with $G$-action such that $\pi_0(R)$ is a finite type $k$-algebra.
\end{thmlist}
\end{exam}

\begin{rem}\label{rem:Iwanari}
In characteristic zero, the Tannaka formalism of Iwanari \cite[Thm.~4.1, Prop.~4.9]{Iwanari} gives an intrinsic characterization of $\sC$ satisfying the condition $(\ast)$.
Namely, it is equivalent to the following two conditions:
\begin{defnlist}
\item
$\hw$ is generated as a symmetric monoidal additive $\infty$-category by a wedge-finite\footnotemark~object $C$ of dimension $d$ and its dual $C^{\vee}$.
\footnotetext{Recall that $C$ is \emph{wedge-finite} of dimension $d\ge 0$ if $\Lambda^{d+1}(C) = 0$ and $\Lambda^d(C)$ is $\otimes$-invertible.
Any wedge-finite object $C$ is dualizable, and we write $C^\vee$ for its dual.
See \cite[Def.~1.1, Rem.~1.4]{Iwanari}.}

\item
The $\Einfty$-ring spectrum
  \begin{equation*}
    A = \bigoplus_{\lambda \in \bZ^d_{\star}} \hw(S_{\lambda}C,\un)\otimes_k S_{\lambda} K
  \end{equation*}
is noetherian and $\pi_0(A)$ is of finite type over $k$.
Here $\bZ^d_{\star}$ is the set of $d$-tuples $(\lambda_1,\cdots,\lambda_d) \in \bZ^d$ such that $\lambda_1 \ge \cdots \ge \lambda_d$,
$S_{\lambda}$ is the Schur functor corresponding to $\lambda$,
and $K$ is the regular representation of $\mrm{GL}_d$.
\end{defnlist}

\end{rem}

\ssec{Gluing weight structures}

As we have seen in \secref{ncmcase} there exist bounded regular $\sE_1$-rings that are not discrete,
given by some triangular matrix rings.
In particular this shows that Conjecture~\ref{mainth} could fail for weight structures that are not
compatible with any symmetric monoidal structure.
It turns out that the conjecture in such generality fails consistently:
we prove a general gluing result for adjacent structures which gives us plenty of counterexamples to the
conjecture as a
byproduct.
For instance, it implies the following statement:

\begin{cor}\label{triangular_regular}
  The triangular matrix \Ering spectrum
  $$ T =
  \begin{pmatrix}
   R & M \\
   0 & S
  \end{pmatrix}
  $$
  is regular, where $R$ and $S$ are regular \Ering spectra and $M$ is a connective $R$-$S$-bimodule that is
  perfect as an $S$-module.
\end{cor}


Recall the following definition from \cite{BlumbergGepnerTabuada}:

\begin{defn}\label{defn:split ses}
A \emph{split short exact sequence} (or semi-orthogonal decomposition) of stable $\infty$-categories is a diagram
  $$
  \begin{tikzcd}
    \sC \arrow[r, shift left = 1, "i_{\sC}"]
      & \sE\arrow[l,shift left = 1, "L_{\sC}"] \arrow[r, shift left = 1, "L_{\sD}"]
      & \sD \arrow[l,shift left = 1, "i_{\sD}"]
  \end{tikzcd}
  $$
satisfying the following conditions:
\begin{defnlist}
  \item
The functors $i_{\sC}$ and $i_{\sD}$ are fully faithful.
  \item
The functor $L_{\sC}$ is right adjoint to $i_{\sC}$, and $i_{\sD}$ is right adjoint to $L_{\sD}$.
  \item
The essential image of $i_{\sD}$ is the right orthogonal to the essential image of $i_{\sC}$.
\end{defnlist}
\end{defn}

\begin{lem}\label{fibseq}
Given a split short exact sequence of stable $\infty$-categories as above, the unit and counit maps of the adjunctions form an exact triangle
  $$i_{\sC}L_{\sC}(X) \to X \to i_{\sD}L_{\sD}(X).$$
\end{lem}
\begin{proof}
Condition (iii) of the definition implies that the fiber of the map $X \to i_{\sD}L_{\sD} X$ has the form $i_{\sC}Y$ for some $Y$.
Now $Y \simeq L_{\sC}i_{\sC}Y \to L_{\sC}X$ is an equivalence.
\end{proof}

The main result of this section is as follows:

\begin{thm}\label{semiorthogonal}
Suppose given a split short exact sequence of stable $\infty$-categories:
  \begin{equation*}
    \begin{tikzcd}
      \sC \arrow[r, shift left = 1, "i_{\sC}"]
      & \sE\arrow[l,shift left = 1, "L_{\sC}"] \arrow[r, shift left = 1, "L_{\sD}"]
      & \sD \arrow[l,shift left = 1, "i_{\sD}"].
    \end{tikzcd}
  \end{equation*}
Assume that $\sC$ and $\sD$ admit weight structures such that $\pi_0\sE(i_{\sD}X, i_{\sC}Y) = 0$ for any $X \in \sD_{w\le 0}$ and $Y \in \Sigma\sC_{w\ge 0}$.
Then there exists a unique weight structure on $\sE$ such that all the functors in the diagram are weight-exact.
Moreover, if $i_{\sD}$ has a right adjoint $R_{\sD}$ and the weight structures on $\sC$ and $\sD$ admit adjacent t-structures, then so does the weight structure on $\sE$.
\end{thm}

The first part of the statement is related to \cite[\S8.2]{BondarkoWeights}.
However, it is not directly applicable here, as the gluing there requires the existence of two extra adjoints while the weight-exactness of all functors is not required.

To deduce Corollary~\ref{triangular_regular} it suffices to set
$\sC = \Mod_R^\perf, \sD = \Mod_S^\perf, \sE = \Mod_T^\perf$ with obvious functors between them. Note that the vanishing condition corresponds to connectivity of $M$ while the existence of a right adjoint $R_{\sD}$ translates as perfectness of $M$.

\ssec{Proof of \thmref{semiorthogonal}}
\label{ssec:proof of semiorthogonal}

\sssec{Weight structure}

We set $\sE_{w\ge 0}$ and $\sE_{w\le 0}$ to be the extension-closures of the unions $i_{\sC}\sC_{w\ge 0} \cup i_{\sD}\sD_{w\ge 0}$
and $i_{\sC}\sC_{w\le 0} \cup i_{\sD}\sD_{w\le 0}$, respectively.
These classes are semi-invariant with respect to translations, and the orthogonality axiom follows from fully faithfulness of $i_{\sC}$ and $i_{\sD}$ and from the extra condition in the statement.
It suffices to construct a weight decomposition for an object $X \in \sE$.

By Lemma \ref{fibseq} we have an exact triangle $i_{\sC}L_{\sC}X \to X \to i_{\sD}L_{\sD}X$.
Take any weight decompositions of $L_{\sC}X$ and $L_{\sD}X$.
Since there are no non-zero maps $\Omega i_{\sD}w_{\le 0} L_{\sD} X \to i_{\sC}w_{\ge 1} L_{\sC}X$ we obtain a homotopy commutative square
$$
\begin{tikzcd}
i_{\sD}w_{\le 0} L_{\sD} \Omega X\arrow[r]\arrow[d] & i_{\sC}w_{\le 0} L_{\sC}X\arrow[d]\\
i_{\sD}L_{\sD} \Omega X\arrow[r] & i_{\sC}Y
\end{tikzcd}
$$
Using Lemma 1.1.11 of \cite{bbd} we get a commutative diagram
$$
\begin{tikzcd}
i_{\sC}w_{\le 0} L_{\sC}X\arrow[d]\arrow[r] &X'\arrow[d]\arrow[r]& i_{\sD}w_{\le 0} L_{\sD}X\arrow[d]\\
i_{\sC}L_{\sC}X \arrow[r]\arrow[d] & X \arrow[r]\arrow[d] &i_{\sD}L_{\sD}X\arrow[d]\\
i_{\sC}w_{\ge 1} L_{\sC}X \arrow[r] &X''\arrow[r]& i_{\sD}w_{\ge 1} L_{\sD}X
\end{tikzcd}
$$
in the homotopy category of $\sE$, whose rows and columns come from exact triangles.
In particular, $X' \in \sE_{w\le 0}$ and $X'' \in \Sigma\sE_{w\ge 0}$, so the exact triangle $X' \to X \to X''$ is a weight decomposition.

\sssec{Adjacent t-structure}

We now assume that the weight structures on $\sC$ and $\sD$ admit adjacent t-structures, and that $i_D$ admits a right adjoint $R_D$.
In this case we will show that the weight structure on $\sE$ constructed above also admits an adjacent t-structure.

It suffices to construct maps $\tau \colon \tau_{\ge 0} X \to X$ for all $X$ where $\tau_{\ge 0} X \in \sE_{w\ge 0}$ and $\pi_0\sE(Y, \Cofib(\tau)) =0$ for any $Y \in \sE_{w\ge 0}$.
First consider the case $X=i_{\sD}Y$.
Then the map $i_{\sD}\tau \colon i_{\sD}\tau_{\ge 0} Y \to i_{\sD}Y$ satisfies the conditions. Indeed,
$$\pi_0\sE(i_{\sD}Z, \Cofib(i_{\sD}\tau))= \pi_0\sE(i_{\sD}Z, i_{\sD}\Cofib(\tau))=\pi_0\sD(Z,\Cofib(\tau))=0$$
for any $Z \in \sD_{w\ge 0}$ and
$$\pi_0\sE(i_{\sC}Z, \Cofib(i_{\sD}\tau)) = \pi_0\sE(i_{\sC}Z, i_{\sD}\Cofib(\tau))= \pi_0\sC(Z, L_{\sC}i_{\sD}\Cofib(\tau)) = 0$$
for any $Z \in \sC$. So, $\pi_0\sE(Z,\Cofib(i_{\sD}\tau))=0$ for any $Z$ from the extension-closure of $i_{\sC}\sC_{w\ge 0}\cup i_{\sD}\sD_{w\ge 0}$.

Next consider the case $X = i_{\sC}Y$.
The cofiber $Y'$ of the map $i_{\sC}\tau \colon i_{\sC}\tau_{\ge 0}Y \to i_{\sC}Y$ satisfies
$$\pi_0\sE(i_{\sC}Z, Y')= \pi_0\sC(Z, \Cofib(\tau))=0.$$
for any $Z \in \sC_{w\ge 0}$.
This orthogonality property is also satisfied by the cofiber $Y''$ of the map $i_{\sD}\tau_{\ge 0}R_{\sD}Y' \to i_{\sD}R_{\sD}Y' \to Y'$ since
$$\pi_0\sE(i_{\sC}Z, i_{\sD}\tau_{\ge 0}R_{\sD}Y')= \pi_0\sE(Z, L_{\sC}i_{\sD}\tau_{\ge 0}R_{\sD}Y') = 0$$
for any $Z\in \sC$.
The map
$$\pi_0\sE(i_{\sD}Z, i_{\sD}\tau_{\ge 0}R_{\sD}Y')\simeq \pi_0\sD(Z, \tau_{\ge 0}R_{\sD}Y') \to \pi_0\sD(Z, R_{\sD}Y') \simeq \pi_0\sE(i_{\sD}Z, Y')$$
is an isomorphism for $Z \in \sD_{w \ge 0}$ and an injection for $Z \in \Omega\sD_{w \ge 0}$ by orthogonality properties of
the t-structure on $\sD$. Therefore from the long exact sequence associated to an exact triangle we see that $\pi_0\sE(i_{\sD}Z,Y'')=0$ for
any $Z \in \sD_{w\ge 0}$. The fiber $F$ of the map $X \to Y''$ is an extension of $i_{\sD}\tau_{\ge 0}R_{\sD}Y'=\Fib(Y' \to Y'')$ by
$i_{\sC}\tau_{\ge 0}Y = \Fib(X \to Y')$, so it belongs to $\sE_{w\ge 0}$ and the map $F \to X$ satisfies the desired conditions.

Finally let $X$ be arbitrary.
By Lemma \ref{fibseq} there is an exact triangle $\Omega i_{\sD}L_{\sD}X \stackrel{p}\to i_{\sC}L_{\sC}X \to X$.
The already constructed maps $\tau_{\ge 0} \Omega i_{\sD}L_{\sD} \to \Omega i_{\sD}L_{\sD}X$ and
$\tau_{\ge 0} i_{\sC}L_{\sC}X \to i_{\sC}L_{\sC}X$ fit into a homotopy commutative diagram
$$
\begin{tikzcd}
\tau_{\ge 0} \Omega i_{\sD}L_{\sD}X \arrow[d, "\tau_{\sD}"]\arrow[r, dotted, "\tau_{\ge 0} p"] & \tau_{\ge 0} i_{\sC}L_{\sC}X\arrow[d,"\tau_{\sC}"]\\
\Omega i_{\sD}L_{\sD}X\arrow[r,"p"]  & i_{\sC}L_{\sC}X.
\end{tikzcd}
$$
Applying Lemma 1.1.11 of \cite{bbd} we get a homotopy commutative diagram
$$
\begin{tikzcd}\label{diagram}
\tau_{\ge 0} \Omega i_{\sD}L_{\sD}X \arrow[d, "\tau_{\sD}"]\arrow[r, "\tau_{\ge 0} p"] & \tau_{\ge 0} i_{\sC}L_{\sC}X\arrow[d,"\tau_{\sC}"] \arrow[r] & P\arrow[d]\\
\Omega i_{\sD}L_{\sD}X\arrow[r,"p"] \arrow[d] & i_{\sC}L_{\sC}X \arrow[r]\arrow[d] & X\arrow[d]\\
\tau_{\le -1}\Omega i_{\sD}L_{\sD}X\arrow[r]  & \tau_{\le -1}i_{\sC}L_{\sC}X\arrow[r]  & N\\
\end{tikzcd}
$$
whose rows and columns are exact triangles. 
We note that $P$ belongs to $\sE_{w\ge 0}$ and $\pi_0\sE(Y, N) =0$ for any $Y \in \sE_{w\ge 1}$. Moreover,
$\tau_{\le -1} \Omega i_{\sD}L_{\sD} X = i_{\sD} \tau_{\le -1} \Omega L_{\sD} X$ by construction of $\tau_{\sD}$, so $\pi_0\sE(Y, N) =0$
for any $Y \in i_{\sC}\sC_{w\ge 0}$.
However,  the group $\pi_0\sE(i_{\sD}Y,N)$ might be non-zero for $Y \in \hw_{\sD}$. The issue is fixed using the following.

{\bf Claim.} There exists an object $I \in \htt_{\sD}$ and a map $\tau_0\colon i_{\sD} I\to N$ such that the map $\pi_0\sE(i_{\sD} Y, \tau_0)$
is an isomorphism for any $Y \in \hw_{\sD}$.

Indeed, assume the claim is proven.
Then it implies the vanishing of $\pi_0\sE(i_{\sD}Y, \Cofib(\tau_0))$ for any $Y \in \sD_{w\ge 0}$.
Moreover, $\pi_0\sE(i_{\sC}Y, \Cofib(\tau_0))$ is just isomorphic to $\pi_0\sE(i_{\sC}Y,N)$ for $Y \in \sC$, so $\pi_0(Y, \Cofib(\tau_0)) = 0$ for
any $Y \in \sE_{w\ge 0}$.
The fiber $\tau_{\ge 0} X$ of the map $X \to \Cofib(\tau_0)$ belongs to $\sE_{w\ge 0}$ as it is an extension of $i_{\sD}I$ by $P$. So $\tau_{\ge 0} X \to X$ gives us the desired map. It suffices to prove the claim.

From the bottom and the middle exact triangles in the diagram above we see that for $Y \in \hw_{\sD}$
$$\pi_0\sE(i_{\sD}Y,N) \simeq im (\pi_0\sE(i_{\sD}Y,X) \stackrel{u}\to \pi_0\sE(i_{\sD} Y,\Sigma \tau_{\le -1}\Omega i_{\sD}L_{\sD}X))$$
and the isomorphism is induced by the map $N \to \Sigma \tau_{\le -1}\Omega i_{\sD}L_{\sD}X$.

Since $i_{\sD}$ is fully faithful, the composition of the unit and the counit $i_{\sD}R_{\sD}X \to X \to i_{\sD}L_{\sD} X$ defines a map
$\gamma \colon R_{\sD}X \to L_{\sD} X$.
We define $I$ to be the image in the abelian category $\htt_{\sD}$ of the map $\pi_0\gamma \colon \pi_0R_{\sD} X \to \pi_0L_{\sD} X$.
There is an obvious map $f$ from $i_{\sD}I$ to $\Sigma \tau_{\le -1}\Omega i_{\sD}L_{\sD}X = i_{\sD}\tau_{\le 0}L_{\sD} X$. The following commutative diagram shows that $f$ maps
$\pi_0\sE(i_{\sD}Y, i_{\sD}I)$ isomorphically to the subset of $\pi_0\sE(i_{\sD} Y, i_{\sD}\tau_{\le 0}L_{\sD} X)$ isomorphic to $\pi_0\sE(i_{\sD}Y,N)$.

$$
\begin{tikzcd}
\pi_0\sE(i_{\sD}Y, i_{\sD} \pi_0R_{\sD}X)\arrow[r, two heads]\arrow[d,"(1)"] & \pi_0\sE(i_{\sD}Y,i_{\sD}I) \arrow[r,hook]\arrow[dr, "f"]& \pi_0\sE(i_{\sD}Y, i_{\sD} \pi_0L_{\sD}X)\arrow[d,"(2)"]\\
\pi_0\sE(i_{\sD}Y, i_{\sD} \tau_{\le 0}R_{\sD}X)\arrow[rr] && \pi_0\sE(i_{\sD}Y, i_{\sD} \tau_{\le 0}L_{\sD}X)\\
\pi_0\sE(i_{\sD}Y, i_{\sD}R_{\sD}X) \arrow[r,"(4)"]\arrow[u,"(3)", swap] & \pi_0\sE(i_{\sD}Y, X) \arrow[r]\arrow[ur,"u"] & \pi_0\sE(i_{\sD}Y,i_{\sD} \tau_{\le 0}L_{\sD}X)\arrow[u,"="]
\end{tikzcd}
$$
Indeed, considering the long exact sequences the maps (1) and (2) fit into, we see that they are isomorphisms by the orthogonality axiom for weight structures. The map (3) is a surjection for the same reason. The map (4) becomes the
map $\pi_0\sD(Y, R_{\sD}i_{\sD}R_{\sD}X) \to \pi_0\sD(Y, R_{\sD}X)$ under the adjunction isomorphism for $i_{\sD}$ and $R_{\sD}$.
The latter map is an isomorphism because $\id_{\sD} \to R_{\sD}i_{\sD}$ is an equivalence.
Therefore we see that the images of $\pi_0\sE(i_{\sD}Y,i_{\sD}I)$ and of $\pi_0\sE(i_{\sD}Y, X)$ in $\pi_0\sE(i_{\sD}Y, i_{\sD} \tau_{\le 0}L_{\sD}X)$ coincide and the former maps injectively onto the image.

Now it suffices to construct a map $i_{\sD}I \to N$ that makes the triangle
$$
\begin{tikzcd}
N \arrow[r] & i_{\sD}\tau_{\le 0}L_{\sD} X\\
i_{\sD}I \arrow[u, dotted]\arrow[ur, "f"]
\end{tikzcd}
$$
commute.
Since $N$ is a fiber of the map $i_{\sD}\tau_{\le 0}L_{\sD} X \to \Sigma \tau_{\le -1} i_{\sC}L_{\sC}X$, it suffices to prove that the
composite map $i_{\sD}I \to \Sigma \tau_{\le -1} i_{\sC}L_{\sC}X$ is null-homotopic. We will show that the element $f$ in the upper left corner
of the following commutative diagram becomes trivial after applying the horizontal map.
$$
\begin{tikzcd}
\pi_0\sE(i_{\sD}I, i_{\sD}\tau_{\le 0}L_{\sD} X) \arrow[r]\arrow[d, "(1)"] & \pi_0\sE(i_{\sD}I, \Sigma \tau_{\le -1} i_{\sC}L_{\sC}X)\arrow[d, "(2)"]\\
\pi_0\sE(i_{\sD}\pi_0R_{\sD}X, i_{\sD}\tau_{\le 0}L_{\sD} X) \arrow[r] & \pi_0\sE(i_{\sD}\pi_0R_{\sD}X, \Sigma \tau_{\le -1} i_{\sC}L_{\sC}X)\\
\pi_0\sE(i_{\sD}\tau_{\le 0}R_{\sD}X, i_{\sD}\tau_{\le 0}L_{\sD} X) \arrow[r]\arrow[u,swap,"(3)"]\arrow[d,"(4)"] & \pi_0\sE(i_{\sD}\tau_{\le 0}R_{\sD}X, \Sigma \tau_{\le -1} i_{\sC}L_{\sC}X)\arrow[u]\arrow[d,"(5)"]\\
\pi_0\sE(i_{\sD}R_{\sD}X, i_{\sD}\tau_{\le 0}L_{\sD} X) \arrow[r] & \pi_0\sE(i_{\sD}R_{\sD}X, \Sigma \tau_{\le -1} i_{\sC}L_{\sC}X)
\end{tikzcd}
$$
The long exact sequence associated to the exact triangle
$$i_{\sD}K \to i_{\sD}\pi_0R_{\sD}X \to i_{\sD}I$$
where $K$ denotes the kernel of the surjection $\pi_0R_{\sD}X \to I$ in the abelian category $\htt_{\sD}$, together with the vanishing
property in
the definition of $\tau_{\sC}$ and $\tau_{\sD}$, yields injectivity of the maps (1) and (2). So it suffices to show that the image of $f$
in $\pi_0\sE(i_{\sD}\pi_0R_{\sD}X, i_{\sD}\tau_{\le 0}L_{\sD} X)$ is mapped to zero via the horizontal map. By definition of $f$ the
image lifts via (3) to a map $\tilde{f}= i_{\sD}\tau_{\le 0}(\gamma)$.
The long exact sequence associated with the exact triangle
$$i_{\sD}\tau_{\ge 1}R_{\sD}X \to i_{\sD}R_{\sD}X \to i_{\sD}\tau_{\le 0}R_{\sD}X$$
together with the vanishing property in the definition of $\tau_{\sC}$ and $\tau_{\sD}$ implies that the maps (4) and (5) are isomorphisms.
So it suffices to show that the image of $\tilde{f}$ via (4) is mapped to zero via the horizontal map.
The image of $\tilde{f}$ is the composite $i_{\sD}R_{\sD}X \to i_{\sD}\tau_{\le 0} R_{\sD} X \to i_{\sD} \tau_{\le 0}L_{\sD} X$.
This factorizes as $i_{\sD}R_{\sD}X \to i_{\sD}L_{\sD}X \to i_{\sD} \tau_{\le 0}L_{\sD} X$.
Therefore, the result follows from the fact that the composite of the two maps in an exact triangle is null-homotopic.

\bibliographystyle{abbrv}

{\small
\bibliography{references}
}

\end{document}